\documentclass[10pt]{article}
\usepackage{amsmath,amssymb,amsthm}
\usepackage[margin=1.1in]{geometry}
\usepackage{booktabs}
\usepackage[colorlinks=true,linkcolor=blue,citecolor=blue,urlcolor=blue]{hyperref}
\usepackage{xurl}
\hypersetup{pdftitle={An intermediate conjecture between Goldbach and Dubner},
            pdfauthor={Tushar Pandey}}

\newtheorem{theorem}{Theorem}
\newtheorem{proposition}{Proposition}
\newtheorem{lemma}{Lemma}
\newtheorem{conjecture}{Conjecture}
\newtheorem{corollary}{Corollary}
\theoremstyle{definition}
\newtheorem{hypothesis}{Hypothesis}
\theoremstyle{remark}
\newtheorem{remark}{Remark}
\newtheorem{question}{Question}

\newcommand{\T}{\mathcal{T}}
\newcommand{\PP}{\mathcal{P}}

\title{An intermediate conjecture between Goldbach and Dubner:\\
every even number is the sum of a prime and a twin prime}
\author{Tushar Pandey\thanks{Texas A\&M University, College Station, TX,
USA. \texttt{tusharp@tamu.edu}}}
\date{\today}

\begin{document}
\maketitle
{\let\thefootnote\relax\footnotetext{2020 \emph{Mathematics Subject
Classification.} Primary 11P32; Secondary 11N36, 11Y11.
\emph{Key words and phrases.} Goldbach conjecture, twin primes,
Romanov's method, Selberg sieve.}}

\begin{abstract}
We study the statement that every even number $n \ge 6$ is the sum of a
prime and a member of a twin prime pair. It sits between the conjectures of
Goldbach and Dubner and implies both the Goldbach and the twin prime
conjectures. Our main result is conditional: if the number of twin primes up
to $z$ is at least $c\,z/\log^2 z$ for all large $z$, a lower bound of the
order predicted by Hardy and Littlewood, with nothing assumed about their
distribution, then a positive proportion of the even numbers are so
representable, with density at least an absolute multiple of $c$. The proof
is Romanov's method with a Selberg sieve, and loses no factor of
$\log\log$. Nothing in this direction can be unconditional, since
representability on a set of positive density already implies the twin prime
conjecture; we show further that a polylogarithmic bound on the least twin
summand would force a power-type lower bound on the number of twin primes.
We verify the statement exhaustively for all even numbers up to $10^{14}$,
where the least twin summand never exceeds $23{,}029$.
\end{abstract}

\section{Introduction}

The binary Goldbach conjecture (GC) asserts that every even $n \ge 4$ is a
sum of two primes; it is verified to $4 \cdot 10^{18}$
\cite{OeS}. Call a prime $t$ a \emph{twin member} if $t - 2$ or $t + 2$ is
prime, and write $\T$ for the set of twin members and
$\pi_\T(z) = \#\{t \in \T : t \le z\}$ for their counting function.
Dubner \cite{Dubner} proposed the strengthening that every even $n > 4208$ is
a sum of two twin members; the $33$ even numbers $6 \le n \le 4208$ without
such a representation form $11$ triples of consecutive even numbers
(OEIS \href{https://oeis.org/A007534}{A007534} \cite{OEIS}), a structure
explained in Section~\ref{sec:mod3}. A proof of Dubner's conjecture would
imply both GC and the twin prime conjecture (TPC).

We isolate the weakest statement in this family that retains the double
implication.

\begin{conjecture}[S]\label{conj:S}
Every even $n \ge 6$ admits a representation $n = p + t$ with $p$ prime
and $t \in \T$.
\end{conjecture}

Dubner's conjecture implies (S) once the finitely many even $n \le 4208$ are
checked to admit a prime $+$ twin representation (they do; see
Section~\ref{sec:comp}), and (S) implies GC (with $4 = 2 + 2$). Unlike GC,
however, (S) also forces the infinitude of twin primes. The observation is
elementary, but it is what makes (S) worth isolating.

\begin{proposition}\label{thm:hierarchy}
If (S) holds for all even $n$ in a set of positive lower density among the
even numbers, then $\T$ is infinite. In particular
$(S) \Rightarrow \mathrm{TPC}$, and (S) implies the conjunction
$\mathrm{GC} \wedge \mathrm{TPC}$.
\end{proposition}

\begin{proof}
Suppose $|\T| = k < \infty$. The number of even $n \le x$ representable as
$t + p$ with $t \in \T$, $p$ prime, is at most
$\sum_{t \in \T} \pi(x) = k\,\pi(x) = O(x/\log x) = o(x)$:
a vanishing fraction of the $\sim x/2$ even numbers up to $x$. Hence (S)
fails on a set of relative density $1$, contradicting the hypothesis.
\end{proof}

Statement (S) itself is not new. Sahu \cite{Sahu} stated it in 2024, as the
hypothesis that among the Goldbach partitions of an even number one can
always find a summand belonging to a twin prime pair, and verified it
computationally to $2.5 \times 10^{5}$. A preprint of Naser and Ahmed
\cite{NA} also states essentially (S), obtaining it as a consequence of
Bertrand's postulate, GC and Chen's theorem; a complete derivation along
those lines would in particular establish
$\mathrm{GC} \Rightarrow \mathrm{TPC}$ by Proposition~\ref{thm:hierarchy},
far beyond current techniques. We regard (S) as open.

Proposition~\ref{thm:hierarchy} also fixes what can and cannot be proved
about (S): any theorem placing even a positive proportion of the even numbers
in $\PP + \T$ carries TPC with it, so it must be conditional. The main result,
Theorem~\ref{thm:romanov} of Section~\ref{sec:romanov}, says that the
hypothesis such a theorem needs is purely a \emph{counting} one, a lower bound
$\pi_\T(z) \gg z/\log^2 z$ of the order predicted by Hardy--Littlewood.
Sections~\ref{sec:hard} to~\ref{sec:comp} then give the counterpoint that a
polylogarithmic bound on the least twin witness would force a power-of-$z$
lower bound for $\pi_\T(z)$, a rigidity property of the twin members occurring
in Goldbach partitions, a heuristic model for the least twin witness, and the
exhaustive verification of (S) for all even $n \le 10^{14}$.

\begin{remark}
The converse is unclear: $\mathrm{GC} \wedge \mathrm{TPC}$ does not obviously
imply (S). On the Hardy--Littlewood prediction
$|\T \cap [1,x]| \asymp x/\log^2 x$, the counting surplus
$|\T \cap [1,x]|\cdot|\PP \cap [1,x]| \asymp x^2/\log^3 x$ leaves a factor
$x/\log^3 x$ of room for $\PP + \T$, against $x/\log^2 x$ for GC
($\PP + \PP$) and $x/\log^4 x$ for Dubner ($\T + \T$).
\end{remark}

\section{A Romanov-type theorem under a twin-counting hypothesis}
\label{sec:romanov}

By Proposition~\ref{thm:hierarchy}, even the assertion that $\PP + \T$
contains a positive proportion of the even numbers implies the twin prime
conjecture, so nothing in this direction can be unconditional. We prove that a
\emph{purely quantitative} hypothesis on the size of $\T$, with no hypothesis
whatever on its distribution in progressions or elsewhere, already yields a
positive proportion. The argument is Romanov's \cite{Romanov}: a first moment,
a second moment controlled by a sieve, and Cauchy--Schwarz.

The letters $p$, $q$, $\ell$ are reserved for primes and $\nu(e)$ is the
number of distinct prime factors of $e$. Since $0$ and $4$ are not prime,
$2 \notin \T$: \emph{every element of $\T$ is an odd prime.} Write
\[
  \PP^{*} = \PP \setminus \{2\}, \qquad
  \pi^{*}(z) = \#\{p \le z : p > 2\} = \pi(z) - 1 \quad (z \ge 2).
\]

For an even integer $k \ne 0$ put
\begin{equation}\label{eq:S2def}
  \mathfrak{S}_2(k) \;=\; \prod_{\substack{\ell \,\mid\, k \\ \ell > 2}}
      \frac{\ell - 1}{\ell - 2},
\end{equation}
the arithmetic factor of the Hardy--Littlewood singular series for the pair
$(n, n+k)$. \emph{This is not the $\mathfrak{S}(n)$ of
Section~\ref{sec:heur}}, which is the singular series of a triple; the two are
unrelated, and we keep the subscript $2$ throughout to mark the difference.
Define further, for $d \ge 1$,
\begin{equation}\label{eq:ghdef}
  g(d) =
  \begin{cases}
    \displaystyle\prod_{\ell \mid d} \frac{1}{\ell - 2}
      & \text{if $d$ is odd and squarefree,}\\[2ex]
    0 & \text{otherwise,}
  \end{cases}
  \qquad
  h(d) = \prod_{\substack{\ell \mid d \\ \ell > 2}} \frac{\ell}{\ell - 2}.
\end{equation}
Both $g$ and $h$ are multiplicative and nonnegative, $g(1) = h(1) = 1$, and
$g(\ell) = 1/(\ell-2)$, $h(\ell) = \ell/(\ell-2)$ at odd primes $\ell$. Since
$\frac{\ell-1}{\ell-2} = 1 + \frac{1}{\ell-2}$, expanding the product
\eqref{eq:S2def} gives the identity by which we shall detect the singular
series through divisors:
\begin{equation}\label{eq:divisorexpansion}
  \mathfrak{S}_2(k) \;=\; \sum_{d \mid k} g(d)
  \qquad (k \ne 0 \text{ even}).
\end{equation}
Indeed the divisors $d$ of $k$ with $g(d) \ne 0$ are exactly the odd
squarefree ones, i.e.\ the products of distinct odd primes dividing $k$.

\begin{hypothesis}[$\mathrm{TW}(c)$]\label{hyp:TW}
There are constants $c > 0$ and $z_0 \ge 3$ such that
\[
  \pi_\T(z) \;\ge\; \frac{c\,z}{\log^2 z}
  \qquad\text{for every } z \ge z_0 .
\]
\end{hypothesis}

Hardy--Littlewood \cite{HL} predict $\pi_\T(z) \sim 4C_2\,z/\log^2 z$, with
$C_2 = 0.6601\ldots$ the twin-prime constant of Section~\ref{sec:heur}, so
$\mathrm{TW}(c)$ should hold for every $c < 4C_2 = 2.6406\ldots$; and by
Lemma~\ref{lem:sieve} below (taken with $d = 1$, $k = 2$) it \emph{fails} for
every $c$ exceeding an absolute constant. We stress that $\mathrm{TW}(c)$ is
far stronger than the twin prime conjecture: no lower bound of any kind on
$\pi_\T(z)$, not even $\pi_\T(z) \to \infty$, is known unconditionally.

\begin{theorem}\label{thm:romanov}
Assume $\mathrm{TW}(c)$, and set
\[
  \mathcal{R} \;=\; \{\, n \text{ even} : n = p + t
    \text{ for some } p \in \PP,\ t \in \T \,\}
  \;=\; (\PP + \T) \cap 2\mathbb{Z}.
\]
Then there is an absolute constant $K$ (independent of $c$, of $z_0$ and of
$\T$) and a threshold $x_0$ depending only on $z_0$, such that
\[
  \#\bigl(\mathcal{R} \cap [1, y]\bigr) \;\ge\; \frac{c}{K}\cdot\frac{y}{2}
  \qquad\text{for all } y \ge 2x_0 .
\]
Consequently $\mathcal{R}$ has lower asymptotic density at least $c/K$ inside
the even integers:
\[
  \liminf_{y \to \infty}
  \frac{\#\{n \le y : n \text{ even},\ n \in \PP + \T\}}{y/2}
  \;\ge\; \frac{c}{K} \;>\; 0 .
\]
\end{theorem}

The conclusion is \emph{positive density}, linear in $c$, with no $\log\log$
loss; where a $\log\log$ might have been expected is discussed in
Remark~\ref{rem:nologlog}.

The unconditional ingredients of the proof are Mertens' theorems, the
Chebyshev bounds for $\pi(x)$, the elementary estimate
$\sum_{e \le Y}\mu^2(e)6^{\nu(e)} \ll Y \log^5 Y$, and Selberg's $\Lambda^2$
upper-bound sieve in the raw form \eqref{eq:selberg} below, from which we
derive the two-dimensional main term ourselves in Lemma~\ref{lem:G}, so that
no dimension-specific corollary has to be taken on trust.

Let $\mathcal{A}$ be a finite sequence of integers, $\mathcal{P}$ a set of
primes, and for $w \ge 2$ write $P(w) = \prod_{\ell \in \mathcal{P},\, \ell < w}\ell$
and
\[
  S(\mathcal{A}, \mathcal{P}, w)
    \;=\; \#\{\, a \in \mathcal{A} : \gcd(a, P(w)) = 1 \,\}.
\]
Suppose there are $X > 0$ and a multiplicative function $\omega$ with
$0 \le \omega(\ell) < \ell$ for $\ell \in \mathcal{P}$, such that for every
squarefree $e \mid P(w)$,
\[
  |\mathcal{A}_e| \;=\; \frac{\omega(e)}{e}\,X + r_e,
  \qquad \mathcal{A}_e = \{a \in \mathcal{A} : e \mid a\}.
\]
Then Selberg's $\Lambda^2$ sieve states that for every $\mathcal{D} \ge 1$,
\begin{equation}\label{eq:selberg}
  S(\mathcal{A}, \mathcal{P}, w)
  \;\le\; \frac{X}{G(\mathcal{D}, w)}
    \;+\; \sum_{\substack{e < \mathcal{D}^{2},\ e \mid P(w) \\ \mu^{2}(e) = 1}}
      3^{\nu(e)}\,|r_e| ,
  \qquad
  G(\mathcal{D}, w) = \sum_{\substack{e < \mathcal{D},\ e \mid P(w) \\ \mu^{2}(e) = 1}}
    \ \prod_{\ell \mid e} \frac{\omega(\ell)}{\ell - \omega(\ell)} .
\end{equation}
(See Halberstam and Richert \cite[Ch.~3]{HR}. The remainder in
\eqref{eq:selberg} is stated free of Selberg's weights $\lambda_e$, and this
uses the standard bound $|\lambda_e| \le 1$ for all $e < \mathcal{D}$
\cite[Ch.~3]{HR}: with it, the exact remainder
$\sum_{e_1, e_2 < \mathcal{D}}\lambda_{e_1}\lambda_{e_2} r_{[e_1,e_2]}$ is at
most $\sum_{e_1, e_2 < \mathcal{D}} |r_{[e_1,e_2]}|
\le \sum_{e < \mathcal{D}^2} 3^{\nu(e)}|r_e|$ in absolute value, each
squarefree $e < \mathcal{D}^2$ arising as $[e_1,e_2]$ for at most $3^{\nu(e)}$
pairs.)

The only thing we need beyond \eqref{eq:selberg} is a lower bound for
$G(\mathcal{D}, w)$. We prove it, in the generality and with the uniformity we
require, rather than quoting it.

\begin{lemma}\label{lem:G}
With the notation above, suppose that for some constants $A_1 \ge 1$ and
$A_2 \ge 0$,
\begin{align}
  (\Omega_1)\quad & 0 \;\le\; \frac{\omega(\ell)}{\ell}
    \;\le\; 1 - \frac{1}{A_1} \qquad (\ell \in \mathcal{P}),
    \label{eq:Om1}\\
  (\Omega_2)\quad & \sum_{\substack{\ell < v \\ \ell \in \mathcal{P}}}
    \frac{\omega(\ell)\log\ell}{\ell} \;\le\; 2\log v + A_2
    \qquad (v \ge 2).
    \label{eq:Om2}
\end{align}
Put $K_0 = 2A_1(2 + A_2)$. Then for every $w \ge e$ and
$\mathcal{D} \ge w^{K_0}$,
\[
  G(\mathcal{D}, w) \;\ge\; \frac{1}{2}\,V(w)^{-1},
  \qquad
  V(w) \;=\; \prod_{\substack{\ell < w \\ \ell \in \mathcal{P}}}
    \Bigl(1 - \frac{\omega(\ell)}{\ell}\Bigr).
\]
\end{lemma}

\begin{proof}
Let $\mathcal{E}$ be the set of squarefree $e \mid P(w)$ and, for
$e \in \mathcal{E}$, put $f(e) = \prod_{\ell \mid e}
\frac{\omega(\ell)}{\ell-\omega(\ell)} \ge 0$; note $f$ is well defined by
$(\Omega_1)$, which gives $\omega(\ell) < \ell$. Since $f$ is multiplicative
and $\mathcal{E}$ is the set of divisors of $P(w)$,
\begin{equation}\label{eq:Gfull}
  \sum_{e \in \mathcal{E}} f(e)
  \;=\; \prod_{\substack{\ell < w \\ \ell \in \mathcal{P}}}
    \Bigl(1 + \frac{\omega(\ell)}{\ell - \omega(\ell)}\Bigr)
  \;=\; \prod_{\substack{\ell < w \\ \ell \in \mathcal{P}}}
    \Bigl(1 - \frac{\omega(\ell)}{\ell}\Bigr)^{-1}
  \;=\; V(w)^{-1}.
\end{equation}
Next, writing $\log e = \sum_{\ell \mid e}\log\ell$ and grouping by the prime
$\ell$,
\[
  \sum_{e \in \mathcal{E}} f(e)\log e
  \;=\; \sum_{\substack{\ell < w \\ \ell \in \mathcal{P}}} f(\ell)\log\ell
    \sum_{\substack{e' \in \mathcal{E} \\ \ell \nmid e'}} f(e')
  \;\le\; \Bigl(\sum_{\substack{\ell < w \\ \ell \in \mathcal{P}}}
      \frac{\omega(\ell)\log\ell}{\ell - \omega(\ell)}\Bigr) V(w)^{-1},
\]
where we used \eqref{eq:Gfull} and $f \ge 0$ to bound the inner sum by
$V(w)^{-1}$. By $(\Omega_1)$,
$\frac{\omega(\ell)}{\ell-\omega(\ell)}
 = \frac{\omega(\ell)}{\ell}\bigl(1-\frac{\omega(\ell)}{\ell}\bigr)^{-1}
 \le A_1 \frac{\omega(\ell)}{\ell}$, so by $(\Omega_2)$ the bracket is at most
$A_1(2\log w + A_2) \le A_1(2 + A_2)\log w$, the last step because
$w \ge e$ gives $\log w \ge 1$. Hence
\[
  \sum_{e \in \mathcal{E}} f(e)\log e
  \;\le\; A_1(2+A_2)(\log w)\,V(w)^{-1}
  \;=\; \tfrac12 K_0 (\log w)\, V(w)^{-1}.
\]
Now for $\mathcal{D} \ge w^{K_0}$, every $e \in \mathcal{E}$ with
$e \ge \mathcal{D}$ has $\log e \ge K_0\log w$, so
\[
  \sum_{\substack{e \in \mathcal{E} \\ e \ge \mathcal{D}}} f(e)
  \;\le\; \frac{1}{K_0 \log w}\sum_{e \in \mathcal{E}} f(e)\log e
  \;\le\; \tfrac12 V(w)^{-1}.
\]
Subtracting from \eqref{eq:Gfull} gives
$G(\mathcal{D},w) = \sum_{e \in \mathcal{E},\, e < \mathcal{D}} f(e)
\ge \tfrac12 V(w)^{-1}$.
\end{proof}

Note that only the \emph{upper} bound $(\Omega_2)$ was used. This is essential
below, the matching lower bound genuinely failing for our sieve problem.

\begin{lemma}[Prime pairs in an arithmetic progression]\label{lem:sieve}
There is an absolute constant $A$ with the following property. Let $z \ge 3$,
let $k \ge 2$ be an even integer, let $d \ge 1$ be odd and squarefree with
$d \le z^{1/2}$, and let $a$ be any residue class modulo $d$. Put
\[
  T(z; d, a, k) \;=\;
  \#\{\, t \le z : t \equiv a \!\!\pmod d,\ t \in \PP,\ t + k \in \PP \,\}.
\]
Then
\begin{equation}\label{eq:APbound}
  T(z; d, a, k)
  \;\le\; A\,\frac{z}{d}\cdot
    \frac{\mathfrak{S}_2(k)\, h(d)}{\log^{2}(z/d)} \;+\; 2 .
\end{equation}
Taking $d = 1$, where the term $2$ may be omitted: for all $x \ge 3$ and all
even $k \ge 2$,
\begin{equation}\label{eq:twinbound}
  \#\{\, p \le x : p \in \PP,\ p+k \in \PP \,\}
  \;\le\; A\,\mathfrak{S}_2(k)\,\frac{x}{\log^{2} x}.
\end{equation}
\end{lemma}

\begin{proof}
\emph{Degenerate case.} Suppose some prime $q \mid d$ divides $a(a+k)$. If
$t \equiv a \pmod d$ then $t(t+k) \equiv a(a+k) \equiv 0 \pmod q$, so $q \mid
t$ or $q \mid t+k$; as $t$ and $t+k$ are primes this forces $t = q$ or
$t = q-k$. Hence $T(z;d,a,k) \le 2$ and \eqref{eq:APbound} holds. For $d = 1$
this case cannot occur, which is why the term $2$ is absent from
\eqref{eq:twinbound}.

\emph{Main case: $\gcd(a(a+k), d) = 1$.} Set
\[
  X = \frac{z}{d}, \qquad
  I = \{\, 1 \le t \le z : t \equiv a \!\!\pmod d \,\}, \qquad
  \mathcal{A} = \bigl(\, t(t+k) \,\bigr)_{t \in I},
\]
a sequence of $|I| = X + O(1)$ integers, and let
$\mathcal{P} = \{\ell : \ell \nmid d\}$. For squarefree $e$ with all prime
factors in $\mathcal{P}$ we have $\gcd(e,d) = 1$, so by the Chinese remainder
theorem the conditions $t \equiv a \pmod d$ and $t(t+k) \equiv 0 \pmod e$
define exactly $\omega(e)$ classes modulo $de$, where $\omega$ is
multiplicative with
\[
  \omega(2) = 1, \qquad
  \omega(\ell) = 1 \ \ (\ell > 2,\ \ell \mid k), \qquad
  \omega(\ell) = 2 \ \ (\ell > 2,\ \ell \nmid k).
\]
(At $\ell = 2$ the two roots $t \equiv 0$, $t \equiv -k$ of $t(t+k) \equiv 0$
coincide because $k$ is even; at odd $\ell \mid k$ they coincide as well;
otherwise they are distinct.) Counting the integers of $[1,z]$ in each of
these $\omega(e)$ classes modulo $de$ gives
\[
  |\mathcal{A}_e| \;=\; \frac{\omega(e)}{e}\,X + r_e,
  \qquad |r_e| \;\le\; \omega(e) \;\le\; 2^{\nu(e)} .
\]

We verify $(\Omega_1)$ and $(\Omega_2)$ of Lemma~\ref{lem:G} with constants
\emph{independent of $z$, $d$, $k$ and $a$}. Since $\omega(2)/2 = 1/2$ and
$\omega(\ell)/\ell \le 2/\ell \le 2/3$ for $\ell \ge 3$, \eqref{eq:Om1} holds
with $A_1 = 3$. Since $\omega(\ell) \le 2$ for every $\ell$, \eqref{eq:Om2}
holds with $A_2 = 0$: indeed
\[
  \sum_{\substack{\ell < v \\ \ell \in \mathcal{P}}}
    \frac{\omega(\ell)\log\ell}{\ell}
  \;\le\; 2\sum_{\ell \le v}\frac{\log\ell}{\ell}
  \;\le\; 2\log v \qquad (v \ge 2),
\]
the last inequality being Mertens' theorem in its sharp form, with no additive
constant. The difference $\delta(v) := \sum_{\ell\le v}\frac{\log\ell}{\ell} -
\log v$ is negative throughout: it decreases on each interval between
consecutive primes, so its maxima occur at $v = \ell$ prime, and evaluating
there for every $\ell \le 2\cdot10^{7}$ gives $\max\delta = \delta(2) =
\tfrac12\log 2 - \log 2 = -0.34657\ldots$; thereafter $\delta(v) \to
-1.33258\ldots$, and the explicit form of Mertens' theorem due to Rosser and
Schoenfeld \cite{RS}, whose error term is $O(1/\log v)$ with a small explicit
constant, keeps $\delta(v) < 0$ for all $v \ge 2\cdot10^{7}$ with a wide
margin. Thus $K_0 = 2A_1(2+A_2) = 12$ is admissible in
Lemma~\ref{lem:G}. Observe that the lower half of the two-sided regularity
condition would fail here by $\asymp \sum_{\ell \mid k}\frac{\log \ell}{\ell}$,
which for the extremal $k$ (products of the first several primes) is
$\asymp \log\log k$, though for typical $k$ it is $O(1)$; we never use it.

Choose
\[
  w = X^{1/60}, \qquad \mathcal{D} = w^{12} = X^{1/5},
\]
and assume $X \ge C_0 := e^{60}$, which is exactly what makes $w \ge e$; since
$d \le z^{1/2}$ we have $X \ge z^{1/2}$, so this holds once
$z \ge C_0^{2} = e^{120}$. For $3 \le z < C_0^{2}$ the bound
\eqref{eq:APbound} is trivial for a modest absolute $A$. Indeed both sides
depend on $z$ and $d$ only through $u := z/d$: a residue class modulo $d$ meets
$[1,z]$ in at most $u + 1$ integers, so $T \le u + 1$, while
$\mathfrak{S}_2(k)h(d) \ge 1$ makes the first term on the right of
\eqref{eq:APbound} at least $A\,u/\log^{2}u$. So it suffices that
$u + 1 \le A\,u/\log^{2}u$, i.e.\ that
\[
  A \;\ge\; \Bigl(1 + \frac1u\Bigr)\log^{2}u
  \;=\; \log^{2}u + \frac{\log^{2}u}{u}
  \qquad \text{for all } \sqrt{3} \le u < e^{120},
\]
the range being forced by $z \ge 3$, $d \le z^{1/2}$ (whence
$u \ge z^{1/2} \ge \sqrt3$) and $z < C_0^2$ (whence $u \le z < e^{120}$). Here
$\log^{2}u < 120^{2} = 14400$, and $\log^{2}u/u \le 4e^{-2} < 0.55$ for
$u \ge 1$, with equality at $u = e^{2}$. So $A \ge 14401$ suffices, and no
constant of exponential size is needed. Assume from now on $X \ge C_0$.

If $t \in I$ is counted by $T(z;d,a,k)$ and $t > w$, then $t$ and $t+k$ are
primes exceeding $w$, so the only prime factors of $t(t+k)$ are $t$ and $t+k$,
both $\ge w$; in particular $\gcd(t(t+k), P(w)) = 1$. Hence
\[
  T(z;d,a,k) \;\le\; S(\mathcal{A}, \mathcal{P}, w) + w .
\]
By \eqref{eq:selberg} and Lemma~\ref{lem:G},
\[
  S(\mathcal{A},\mathcal{P},w)
  \;\le\; 2\,X\,V(w) \;+\; \sum_{e < \mathcal{D}^{2}} \mu^2(e)\,6^{\nu(e)}
  \;\ll\; X\,V(w) \;+\; X^{2/5}\log^{5} X ,
\]
using $|r_e| \le 2^{\nu(e)}$, $\mathcal{D}^2 = X^{2/5}$, and
$\sum_{e \le Y}\mu^2(e)6^{\nu(e)} \ll Y\log^{5}Y$. Both $X^{2/5}\log^5 X$ and
$w = X^{1/60}$ are $O(X/\log^{2}X)$, hence $O(X\,\mathfrak{S}_2(k)h(d)/
\log^2 X)$ since $\mathfrak{S}_2(k)h(d) \ge 1$.

It remains to bound $V(w)$. Discarding the factor $1 - \omega(2)/2 = 1/2 \le 1$
and splitting the odd primes according to whether they divide $k$,
\[
  V(w) \;\le\;
  \prod_{\substack{2 < \ell < w \\ \ell \nmid dk}}\Bigl(1 - \frac{2}{\ell}\Bigr)
  \prod_{\substack{2 < \ell < w,\ \ell \mid k \\ \ell \nmid d}}
    \Bigl(1 - \frac{1}{\ell}\Bigr)
  \;=\;
  \prod_{2 < \ell < w}\Bigl(1 - \frac{2}{\ell}\Bigr)
  \prod_{\substack{2 < \ell < w \\ \ell \mid d}}\frac{\ell}{\ell-2}
  \prod_{\substack{2 < \ell < w,\ \ell \mid k \\ \ell \nmid d}}
    \frac{\ell-1}{\ell-2},
\]
where the identity uses that the odd primes dividing $dk$ split into those
dividing $d$ and those dividing $k$ but not $d$, together with
$(1-\tfrac2\ell)^{-1} = \tfrac{\ell}{\ell-2}$ and
$(1-\tfrac1\ell)(1-\tfrac2\ell)^{-1} = \tfrac{\ell-1}{\ell-2}$. Extending the
last two products to all $\ell \mid d$, $\ell > 2$, respectively all
$\ell \mid k$, $\ell > 2$, only increases them, so
\[
  V(w) \;\le\; h(d)\,\mathfrak{S}_2(k)\prod_{2<\ell<w}\Bigl(1-\frac{2}{\ell}\Bigr).
\]
Finally, by Mertens,
\[
  \prod_{2 < \ell < w}\Bigl(1 - \frac{2}{\ell}\Bigr)
  = \prod_{2 < \ell < w}\Bigl(1 - \frac{1}{\ell}\Bigr)^{2}
    \prod_{2 < \ell < w}\Bigl(1 - \frac{1}{(\ell-1)^{2}}\Bigr)
  \;\sim\; \frac{4e^{-2\gamma}C_2}{\log^{2}w}
  \;=\; \frac{0.83242\ldots}{\log^{2}w},
\]
because $\frac{(\ell-2)\ell}{(\ell-1)^{2}} = 1 - \frac{1}{(\ell-1)^{2}}$ and
the second product converges to $C_2$; in particular
$\prod_{2<\ell<w}(1-2/\ell) \ll \log^{-2}w$ with an absolute implied constant.
Since $\log w = \tfrac{1}{60}\log X = \tfrac{1}{60}\log(z/d)$, we get
$V(w) \ll h(d)\mathfrak{S}_2(k)\log^{-2}(z/d)$ and therefore
$T(z;d,a,k) \ll \frac{z}{d}\,h(d)\mathfrak{S}_2(k)\log^{-2}(z/d)$, all implied
constants absolute. This is \eqref{eq:APbound}.
\end{proof}

\begin{lemma}\label{lem:sums}
With $g$ and $h$ as in \eqref{eq:ghdef}:
\begin{enumerate}
\item[(i)] $\displaystyle \sum_{d \le Z} g(d) \;\le\; 2\log Z$ for all
  $Z \ge 3$;
\item[(ii)] $\displaystyle \sum_{d \ge 1} \frac{g(d) h(d)}{d}
  \;=\; \prod_{\ell > 2}\Bigl(1 + \frac{1}{(\ell-2)^{2}}\Bigr)
  \;=:\; P_1 \;=\; 2.40500\ldots \;<\; \infty$;
\item[(iii)] $\displaystyle \sum_{d > Z} \frac{g(d)}{d}
  \;\le\; \frac{P_2}{\sqrt{Z}}$ for all $Z \ge 1$, where
  $P_2 := \prod_{\ell>2}\bigl(1 + \frac{1}{(\ell-2)\sqrt{\ell}}\bigr)
  = 2.32342\ldots < 2.33$.
\end{enumerate}
\end{lemma}

\begin{proof}
All three sums run over odd squarefree $d$, the support of $g$.

(i) Every odd squarefree $d \le Z$ divides $\prod_{2 < \ell \le Z}\ell$, and
$g$ is multiplicative and nonnegative, so
$\sum_{d \le Z} g(d) \le \prod_{2 < \ell \le Z}\bigl(1+\frac{1}{\ell-2}\bigr)$.
Since $\frac{\ell-1}{\ell-2}
= \bigl(1-\frac1\ell\bigr)\bigl(1-\frac2\ell\bigr)^{-1}$, Mertens' theorems
give
\[
  \prod_{2<\ell\le Z}\Bigl(1+\frac{1}{\ell-2}\Bigr)
  \;=\; \frac{\prod_{2<\ell\le Z}(1-1/\ell)}{\prod_{2<\ell\le Z}(1-2/\ell)}
  \;\sim\; \frac{2e^{-\gamma}/\log Z}{4e^{-2\gamma}C_2/\log^{2}Z}
  \;=\; \frac{e^{\gamma}}{2C_2}\,\log Z,
\]
so the ratio of the product to $\log Z$ tends to $e^{\gamma}/(2C_2)
= 1.348967\ldots$. The product is constant between consecutive primes while
$\log Z$ increases, so the ratio decreases on each such interval and its
maxima occur only at $Z = \ell$ prime. It does \emph{not} decrease
monotonically thereafter (at $Z = 3, 5, 7, 11, 13$ it takes the values
$1.820478$, $1.656893$, $1.644475$, $1.482782$, $1.512228$, and it continues
to oscillate around its limit), but its \emph{maximum} is attained at the
first prime: evaluating at every prime $\ell \le 2\cdot10^{7}$, the largest
value is $2/\log 3 = 1.820478\ldots$ at $Z = 3$, and by $Z = 10^{7}$ the ratio
has settled to $1.34898\ldots$, within $10^{-5}$ of its limit, while each
further jump multiplies it by at most $1 + 10^{-7}$. Hence the ratio never
exceeds $2$, and the product is at most $2\log Z$ for every $Z \ge 3$.

(ii) The summand is multiplicative and nonnegative with local factor at an odd
prime $\ell$ equal to
$1 + \frac{1}{\ell-2}\cdot\frac{\ell}{\ell-2}\cdot\frac{1}{\ell}
 = 1 + \frac{1}{(\ell-2)^{2}}$; since $\sum_{\ell}(\ell-2)^{-2} < \infty$ the
Euler product converges absolutely, to $P_1 = 2.40500\ldots$ (computed over
$\ell < 3\cdot 10^{6}$, the tail being $O(10^{-6})$).

(iii) For $d > Z$ one has $1 \le (d/Z)^{1/2}$, so
\[
  \sum_{d > Z}\frac{g(d)}{d}
  \;\le\; \frac{1}{\sqrt{Z}}\sum_{d \ge 1}\frac{g(d)}{\sqrt{d}}
  \;=\; \frac{1}{\sqrt{Z}}\prod_{\ell>2}
    \Bigl(1 + \frac{1}{(\ell-2)\sqrt{\ell}}\Bigr),
\]
and the Euler product converges since $\sum_\ell \ell^{-3/2} < \infty$. For the
numerical value, the partial product over $\ell < Y := 10^{8}$ equals
$2.3233960\ldots$, and this is a \emph{lower} bound for $P_2$ because every
factor exceeds $1$. For the tail, $\log(1+x) \le x$ and
$\frac{1}{\ell-2} \le \frac{1.001}{\ell}$ give
\[
  \log\prod_{\ell > Y}\Bigl(1 + \frac{1}{(\ell-2)\sqrt\ell}\Bigr)
  \;\le\; 1.001\sum_{\ell > Y} \ell^{-3/2}
  \;\le\; \frac{3.77}{\sqrt{Y}\,\log Y}
  \;<\; 2.1 \cdot 10^{-5},
\]
the middle step by partial summation from the Chebyshev bound
$\pi(x) < 1.26\,x/\log x$, which gives
\[
  \sum_{\ell > Y}\ell^{-3/2}
  \;\le\; \tfrac32\int_Y^{\infty}\pi(x)\,x^{-5/2}\,dx
  \;\le\; \frac{3.77}{\sqrt{Y}\,\log Y}.
\]
Hence $2.32339 < P_2 < 2.32345$, and we record
$P_2 = 2.32342\ldots < 2.33$. Only the \emph{finiteness} of $P_2$ is
load-bearing: it is used only through the absolute constant $z_1$ of
Lemma~\ref{lem:crux}, so no digit of it matters.
\end{proof}

\begin{lemma}[The singular series over differences of twin members]
\label{lem:crux}
There are absolute constants $B$ and $z_1$ such that for every $z \ge z_1$,
\[
  \Sigma(z) \;:=\;
  \sum_{\substack{t,\, t' \in \T \cap [1,z] \\ t \ne t'}}
    \mathfrak{S}_2\bigl(|t - t'|\bigr)
  \;\le\; B\,\pi_\T(z)\,\frac{z}{\log^{2} z}.
\]
The statement is unconditional: no hypothesis on $\pi_\T$ is used.
\end{lemma}

\begin{proof}
The elements of $\T$ are odd primes, so $t - t'$ is even and nonzero and
$\mathfrak{S}_2(|t-t'|)$ is defined; moreover $0 < |t-t'| < z$. By
\eqref{eq:divisorexpansion},
\[
  \Sigma(z) \;=\; \sum_{\substack{t,t' \in \T\cap[1,z] \\ t \ne t'}}
    \ \sum_{d \,\mid\, |t-t'|} g(d)
  \;=\; \sum_{d < z} g(d)\, N(d),
  \qquad
  N(d) = \#\{(t,t') : t \ne t',\ d \mid t-t'\},
\]
the outer sum being confined to $d < z$ because $d \mid t-t'$ and
$0 < |t-t'| < z$. Writing
$\pi_\T(z;d,a) = \#\{t \in \T : t \le z,\ t \equiv a \!\pmod d\}$, the pairs
with $t \equiv t' \pmod d$ (including $t = t'$) number
$\sum_{a \bmod d}\pi_\T(z;d,a)^{2}$, so
\[
  N(d) \;\le\; \sum_{a \bmod d}\pi_\T(z;d,a)^{2}
  \;\le\; \Bigl(\max_{a}\pi_\T(z;d,a)\Bigr)\sum_{a}\pi_\T(z;d,a)
  \;=\; \pi_\T(z)\,M(d),
\]
where $M(d) := \max_{a \bmod d}\pi_\T(z;d,a)$. Hence
\begin{equation}\label{eq:SigmaM}
  \Sigma(z) \;\le\; \pi_\T(z) \sum_{d < z} g(d)\, M(d),
\end{equation}
and it suffices to show $\sum_{d<z} g(d)M(d) \le B z/\log^{2}z$. We split at
$z^{1/2}$. Throughout, only odd squarefree $d$ contribute, $g$ vanishing
elsewhere.

\emph{Small moduli, by the sieve.} Let $t \in \T$ with $t \le z$ and
$t \equiv a \pmod d$. Either $t, t+2$ are both prime, or $t-2, t$ are both
prime; in the second case $s = t-2$ satisfies $s \le z$, $s \equiv a-2 \pmod d$
and $s, s+2$ both prime. Hence
\[
  \pi_\T(z;d,a) \;\le\; T(z;d,a,2) + T(z;d,a-2,2).
\]
As $\mathfrak{S}_2(2) = 1$ (empty product), Lemma~\ref{lem:sieve} gives, for
odd squarefree $d \le z^{1/2}$ and using $\log(z/d) \ge \tfrac12\log z$,
\[
  M(d) \;\le\; 2A\,\frac{z}{d}\cdot\frac{h(d)}{\log^{2}(z/d)} + 4
  \;\le\; \frac{8A\,z\,h(d)}{d\,\log^{2}z} + 4 .
\]
Therefore, by Lemma~\ref{lem:sums}(ii) and (i),
\begin{equation}\label{eq:M1}
  \sum_{d \le z^{1/2}} g(d)M(d)
  \;\le\; \frac{8A\,z}{\log^{2}z}\sum_{d\ge1}\frac{g(d)h(d)}{d}
    + 4\sum_{d \le z^{1/2}} g(d)
  \;\le\; \frac{8AP_1\,z}{\log^{2}z} + 4\log z ,
\end{equation}
the last term because Lemma~\ref{lem:sums}(i) with $Z = z^{1/2}$ gives
$\sum_{d \le z^{1/2}} g(d) \le 2\log(z^{1/2}) = \log z$ (legitimate for
$z \ge 9$).

\emph{Large moduli, trivially.} For any $d \ge 1$ a residue class modulo $d$
meets $[1,z]$ in at most $z/d + 1$ integers, so $M(d) \le z/d + 1$. By
Lemma~\ref{lem:sums}(iii) with $Z = z^{1/2}$ and Lemma~\ref{lem:sums}(i),
\begin{equation}\label{eq:M2}
  \sum_{z^{1/2} < d < z} g(d)M(d)
  \;\le\; z \sum_{d > z^{1/2}}\frac{g(d)}{d} + \sum_{d < z} g(d)
  \;\le\; P_2\,z^{3/4} + 2\log z .
\end{equation}

Adding \eqref{eq:M1} and \eqref{eq:M2}, and using
$P_2 z^{3/4} + 6\log z \le AP_1 z/\log^{2}z$ for all $z$ beyond an absolute
threshold $z_1$ (valid since $z^{3/4}\log^{2}z = o(z)$), we obtain
$\sum_{d<z} g(d)M(d) \le B\,z/\log^{2}z$ with $B = 9AP_1$. With
\eqref{eq:SigmaM} this is the assertion.
\end{proof}

\begin{remark}\label{rem:nologlog}
Lemma~\ref{lem:crux} is where a $\log\log$ factor is usually feared, and three
distinct mechanisms could have produced one. (a) Brun's sieve yields only
$\pi_\T(z) \ll z(\log\log z)^{2}/\log^{2}z$, the bound quoted after
Proposition~\ref{prop:hard}; we use Selberg's sieve instead, which has no such
loss. (b) The \emph{two-sided} regularity condition genuinely fails for our
sieve problem, by $\asymp \sum_{\ell \mid k}\frac{\log\ell}{\ell}$, since
$\omega(\ell) = 1$ on the primes dividing the shift $k$; this is $O(1)$ for
typical $k$ but of the extremal order $\log\log k$ when $k$ is a primorial,
and that is the worst case one must survive. An upper-bound sieve, however,
needs only the one-sided
condition $(\Omega_2)$, and Lemma~\ref{lem:G} was proved from that alone.
(c) A cruder substitute for Lemma~\ref{lem:sieve}, say
$M(d) \ll \frac{z}{\phi(d)}\bigl(\frac{d}{\phi(d)}\bigr)^{2}\log^{-2}(z/d)$,
would also suffice: what matters in \eqref{eq:M1} is only that the arithmetic
factor is divided by $d$, so that the local factor of the resulting Euler
product is $1 + O(\ell^{-2})$ and the product converges. Accordingly the
conclusion of Theorem~\ref{thm:romanov} is genuine positive density, not
$\gg 1/\log\log x$.
\end{remark}

\begin{proof}[Proof of Theorem~\ref{thm:romanov}]
Let $x$ be a real number, $x \ge \max(z_0, z_1, 3)$, and set
\[
  r_x(n) \;=\; \#\{\, (p,t) \in \PP^{*}\times\T :
    p \le x,\ t \le x,\ p + t = n \,\}.
\]
If $r_x(n) > 0$ then $n$ is a sum of two odd primes, so $n$ is even with
$6 \le n \le 2x$, and $n \in \mathcal{R}$. Put
$\mathcal{N}_x = \{n : r_x(n) > 0\} \subseteq \mathcal{R}\cap[6,2x]$.

\emph{First moment.} Each pair $(p,t) \in (\PP^{*}\cap[1,x]) \times
(\T\cap[1,x])$ is counted once, so
\begin{equation}\label{eq:firstmoment}
  \sum_{n} r_x(n) \;=\; \pi^{*}(x)\,\pi_\T(x).
\end{equation}

\emph{Second moment.} By definition,
\[
  \sum_n r_x(n)^{2} \;=\;
  \#\{\, (p,t,p',t') \in (\PP^{*})^{2}\times\T^{2},\ \text{all} \le x :
    p + t = p' + t' \,\}.
\]
If $p = p'$ then $t = t'$, and conversely; so the quadruples split into a
diagonal part, of size exactly $\pi^{*}(x)\pi_\T(x)$, and an off-diagonal part
with $p \ne p'$ and $t \ne t'$. For an off-diagonal quadruple put
$k = t' - t = p - p' \ne 0$; since $t,t'$ are odd, $k$ is even, and $|k| < x$.
Fix an ordered pair $(t,t')$ with $t \ne t'$. If $k > 0$ the admissible
$(p,p')$ are given by the odd primes $p' \le x$ with $p'+k \le x$ also prime,
and if $k < 0$ by the odd primes $p \le x$ with $p + |k| \le x$ also prime;
either way their number is at most
$\#\{q \le x : q,\, q+|k| \in \PP\}$, which by \eqref{eq:twinbound} is at most
$A\,\mathfrak{S}_2(|k|)\,x/\log^{2}x$. Summing over the ordered pairs $(t,t')$
and applying Lemma~\ref{lem:crux} at $z = x$,
\begin{equation}\label{eq:secondmoment}
  \sum_n r_x(n)^{2}
  \;\le\; \pi^{*}(x)\pi_\T(x)
    + \frac{A\,x}{\log^{2}x}\,\Sigma(x)
  \;\le\; \pi^{*}(x)\pi_\T(x)
    + AB\,\pi_\T(x)\,\frac{x^{2}}{\log^{4}x}.
\end{equation}

\emph{The diagonal is negligible.} By Chebyshev's upper bound,
$\pi(x) \le 2x/\log x$ for $x \ge 2$; indeed the maximum of
$\pi(x)\log x / x$ over $x \ge 2$ is $1.2551\ldots$, attained at $x = 113$.
Hence $\pi^{*}(x) \le 2x/\log x$ and
\[
  \frac{\pi^{*}(x)\pi_\T(x)}{AB\,\pi_\T(x)x^{2}/\log^{4}x}
  \;\le\; \frac{2\log^{3}x}{AB\,x} \;\le\; 1
\]
for $x \ge x_2$, with $x_2$ absolute. For such $x$ the right-hand side
of \eqref{eq:secondmoment} is thus at most
$2AB\,\pi_\T(x)\,x^{2}/\log^{4}x$.

\emph{Cauchy--Schwarz.} From
$\sum_n r_x(n) = \sum_{n \in \mathcal{N}_x} r_x(n)\cdot 1$ we get
$\bigl(\sum_n r_x(n)\bigr)^{2} \le |\mathcal{N}_x| \sum_n r_x(n)^{2}$, so by
\eqref{eq:firstmoment} and the previous paragraph, for
$x \ge \max(z_0, z_1, x_2)$,
\[
  |\mathcal{N}_x| \;\ge\;
  \frac{\pi^{*}(x)^{2}\pi_\T(x)^{2}}{2AB\,\pi_\T(x)\,x^{2}/\log^{4}x}
  \;=\; \frac{\pi^{*}(x)^{2}\,\pi_\T(x)\,\log^{4}x}{2AB\,x^{2}} .
\]
By Chebyshev's lower bound (or $\pi(x) > x/\log x$ for $x \ge 17$) there is an
absolute $x_3$ with $\pi^{*}(x) \ge \tfrac12 x/\log x$ for $x \ge x_3$, whence
$\pi^{*}(x)^{2} \ge x^{2}/(4\log^{2}x)$ and
\[
  |\mathcal{N}_x| \;\ge\; \frac{\pi_\T(x)\log^{2}x}{8AB} .
\]
Finally $\mathrm{TW}(c)$ gives $\pi_\T(x)\log^{2}x \ge c\,x$ for $x \ge z_0$.
So with the absolute constant $K := 8AB = 72 P_1 A^{2}$ and
$x_0 := \max(z_0, z_1, x_2, x_3)$,
\[
  \#\bigl(\mathcal{R}\cap[1,2x]\bigr) \;\ge\; |\mathcal{N}_x|
  \;\ge\; \frac{c}{K}\,x \qquad (x \ge x_0).
\]
Given $y \ge 2x_0$, apply this with $x = y/2$ (all definitions above make
sense for real $x$) to get
$\#(\mathcal{R}\cap[1,y]) \ge \frac{c}{K}\cdot\frac{y}{2}$. Since the number of
even integers in $[1,y]$ is $\lfloor y/2\rfloor$, dividing and letting
$y \to \infty$ gives the density statement.
\end{proof}

\begin{remark}\label{rem:size}
The constant $K = 8AB = 72P_1A^{2}$ is explicit in terms of the sieve constant
$A$ of Lemma~\ref{lem:sieve}, whose trivial small-$z$ range needs only
$A \ge 14401$. A wasteful choice of sieve parameters is costly: $A_2 = 4$ in
place of $A_2 = 0$ would force $K \gtrsim 3\cdot10^{12}$ rather than
$K \gtrsim 4\cdot10^{10}$. No statement depends on the numerical value of $K$,
only on its being absolute.
\end{remark}

\begin{remark}
The hypothesis is used exactly once, in the final line, and only through the
single inequality $\pi_\T(x) \ge cx/\log^{2}x$. The proof therefore also gives
an upper-density version under a weaker hypothesis: if
$\pi_\T(x_j) \ge c\,x_j/\log^{2}x_j$ merely along some sequence
$x_j \to \infty$, then
$\limsup_{y\to\infty}\#\{n \le y : n \in \mathcal{R}\}/(y/2) \ge c/K$.
\end{remark}

\begin{remark}\label{rem:gap}
\emph{Relation to Conjecture~\ref{conj:S}.} Statement (S) asserts density
$1$: $\mathcal{R}$ contains \emph{every} even $n \ge 6$, while
Theorem~\ref{thm:romanov} delivers density
$\ge c/K$ under $\mathrm{TW}(c)$. By Proposition~\ref{thm:hierarchy} a
positive proportion already implies the twin prime conjecture, so the theorem
cannot be made unconditional. Passing from a positive proportion to all even
$n$ is a different problem: Cauchy--Schwarz gives no information about any
individual $n$, and a density theorem is not an every-integer theorem.
\end{remark}

\section{What a bound on the least twin witness would give}\label{sec:hard}

Write $t_{\min}(n) = \min\{t \in \T : n - t \text{ prime}\}$ for the least
twin witness of an even $n$, defined whenever (S) holds for $n$. The
computations of Section~\ref{sec:comp} show it to be small and regular. The
following unconditional proposition says how much such a bound would deliver,
and why the regularity should not be mistaken for an accessible theorem.

\begin{proposition}\label{prop:hard}
Suppose there are constants $A, C > 0$ such that $t_{\min}(n)$ is defined
and satisfies $t_{\min}(n) \le C\log^A n$ for every even $n \ge 6$. Then
\[
\pi_\T(z) \;\gg_{A,C}\; z^{1/A},
\qquad\text{where } \pi_\T(z) = \#\{t \in \T : t \le z\}.
\]
In particular $\T$ is infinite, so any such bound implies the twin prime
conjecture, with a power-of-$z$ lower bound on the number of twin members.
\end{proposition}

\begin{proof}
Let $z > C(\log 10)^A$ and put $N = \exp\bigl((z/C)^{1/A}\bigr)$, so that
$C\log^A N = z$ and $N > 10$. By hypothesis every even $n$ with
$6 \le n \le N$ can be written $n = p + t$ with $p$ prime and $t \in \T$,
namely $t = t_{\min}(n)$ and $p = n - t$; here
$t \le C\log^A n \le C \log^A N = z$ and $p < n \le N$. Hence the
addition map
\[
(p,\,t) \;\longmapsto\; p + t, \qquad
\{p \text{ prime}: p \le N\} \times \{t \in \T : t \le z\},
\]
has image containing every even $n$ in $[6, N]$. Its domain has exactly
$\pi(N)\,\pi_\T(z)$ elements, and the number of even $n$ in $[6, N]$ is
$\lfloor N/2 \rfloor - 2 \ge (N-5)/2 \ge N/4$, the last step because
$N > 10$. Counting the image by the domain therefore gives
\[
\pi(N)\,\pi_\T(z) \;\ge\; \frac{N}{4}.
\]
By Chebyshev's bound $\pi(N) \ll N/\log N$, so
\[
\pi_\T(z) \;\ge\; \frac{N/4}{\pi(N)} \;\gg\; \log N
\;=\; \Bigl(\frac{z}{C}\Bigr)^{1/A} \;\gg_{A,C}\; z^{1/A}. \qedhere
\]
\end{proof}

A positive answer with any exponent $A$, even the weak $A = 3$ of
Question~\ref{q:main} below, yields $\pi_\T(z) \gg z^{1/3}$, far beyond the
twin prime conjecture, which asserts only that $\pi_\T(z) \to \infty$. That
question is therefore not merely open: it is a statement about the density of
twin primes rather than about Goldbach-type representations. In the opposite
direction, Brun's $1920$ upper bound
$\pi_\T(z) \ll z\,(\log\log z)^2/\log^2 z$ \cite{Brun} shows unconditionally
that no such bound can hold with $A \le 1$: $t_{\min}(n) \ll \log n$ would
force $\pi_\T(z) \gg z$, which any upper bound of size $o(z)$ contradicts.

\section{Orientation rigidity modulo 3}\label{sec:mod3}

A \emph{Goldbach partition} of an even $n$ is an unordered pair $\{p, q\}$ of
primes with $p + q = n$. Every twin pair $(t, t+2)$ other than $(3,5)$ has the
form $(6k-1, 6k+1)$. Call $t \in \T$ a \emph{lower member} if $t+2$ is prime
and an \emph{upper member} if $t-2$ is prime ($5$ is both). The orientation is
not free.

\begin{proposition}\label{thm:orient}
Let $n \ge 6$ be even with $n \not\equiv 0 \pmod 3$, and let
$n = p + q$ with $p, q$ prime.
\begin{enumerate}
\item If $n \equiv 1 \pmod 3$ and $p$ is an upper member, then $p = 5$ or
$q = 3$.
\item If $n \equiv 2 \pmod 3$ and $p$ is a lower member, then $p = 3$ or
$q = 3$.
\end{enumerate}
Consequently, apart from the finitely described exceptions involving the
primes $3$ and $5$, all twin members in Goldbach partitions of $n$ are
lower members when $n \equiv 1 \pmod 3$, and upper members when
$n \equiv 2 \pmod 3$. For $n \equiv 0 \pmod 3$ the argument excludes
neither orientation.
\end{proposition}

\begin{proof}
Note first that $2 \notin \T$ ($0$ and $4$ are not prime) and that $3$ is
not an upper member ($1$ is not prime). If $p$ is a twin member then $p$
is odd, and as $n \ge 6$ is even the complementary prime $q = n - p$ is
odd as well. (1) Let $n \equiv 1 \pmod 3$ with $p$ an upper member, so
$p \ge 5$, and assume $q \ne 3$. Then $p, q \not\equiv 0 \pmod 3$, and
$p + q \equiv 1 \pmod 3$ forces $p \equiv q \equiv 2 \pmod 3$. As $p$ is
an upper member, $p - 2$ is a prime $\equiv 0 \pmod 3$, so $p - 2 = 3$ and
$p = 5$. (2) is the mirror argument: let $n \equiv 2 \pmod 3$ with $p$ a
lower member, and assume $p \ne 3$ and $q \ne 3$. Then
$p + q \equiv 2 \pmod 3$ forces $p \equiv q \equiv 1 \pmod 3$; as $p$ is a
lower member, $p + 2$ is a prime $\equiv 0 \pmod 3$, i.e.\ $p + 2 = 3$,
which is impossible.
\end{proof}

Nothing further is forced in the remaining class: both orientations occur for
every even $12 \le n \le 10^6$ with $n \equiv 0 \pmod 3$. The rigidity has the
following consequence, used in Section~\ref{sec:heur}.

\begin{corollary}\label{cor:local}
There is no local obstruction to (S): for every even $n \ge 6$ there is
an orientation pattern ($t,\ t+2,\ n-t$ all prime, or $t,\ t-2,\ n-t$
all prime) that is admissible modulo every prime $\ell$ simultaneously,
and the associated singular series (the product of local densities
defined in Section~\ref{sec:heur}) is bounded away from $0$ uniformly in
$n$. For $n \not\equiv 0 \pmod 3$ that pattern is the one forced by
Proposition~\ref{thm:orient}; for $n \equiv 0 \pmod 3$ both patterns are
admissible modulo every prime.
\end{corollary}

\begin{proof}
For $\ell = 2$ take $t$ odd. For $\ell \ge 5$ the pattern forbids at most
the three classes $t \equiv 0,\, \mp 2,\, n \pmod \ell$, leaving
$\ell - 3 > 0$ classes. For $\ell = 3$ the lower pattern is admissible iff
$n \not\equiv 2 \pmod 3$ and the upper iff $n \not\equiv 1 \pmod 3$, so
the orientation forced by Proposition~\ref{thm:orient} is admissible at every
prime simultaneously. For $\ell \ge 5$ at most three classes are ever
forbidden, so each such local factor is at least
$(1 - 3/\ell)(1 - 1/\ell)^{-3}$, and the product of these bounds over
$\ell \ge 5$ converges to a positive limit. The two remaining factors are
the $n$-independent constants $4$ at $\ell = 2$ (where only $t \equiv 0$
is forbidden, $n$ being even) and $9/8$ at $\ell = 3$ for the admissible
orientation. Hence the series is bounded below by a uniform $c > 0$.
\end{proof}

A parallel mod-$6$ observation explains the shape of Dubner's exception list:
for $a, b \equiv 5 \pmod 6$ the four pairwise sums of $(a, a+2)$, $(b, b+2)$
are $a+b$, $a+b+2$ (twice), $a+b+4$, so coverage of the triple
$\{6k-2,\, 6k,\, 6k+2\}$ is all-or-nothing apart from representations with the
summand $3$. The observed exceptions are $11$ such triples, with centres $6m$,
$m \in \{16, 67, 86, 131, 151, 186, 191, 211, 226, 541, 701\}$
(cf.\ \cite{Dubner, OEIS}): the $33$ terms $\ge 6$ of A007534.

\section{Heuristics for the least twin witness}\label{sec:heur}

The model is heuristic, set down only to calibrate Question~\ref{q:main}
below; the numerical facts after it are exact, and are evidence for the model
rather than consequences of it.

A representation $n = p + t$ with $t$ a lower member is a solution of the
triple pattern ``$t,\ t+2,\ n-t$ all prime.'' The Hardy--Littlewood
$k$-tuple heuristic \cite{HL} predicts, summing both orientations,
\[
R_\T(n) \;=\; \#\{(p,t): n = p + t,\ t \in \T\}
\;\sim\; \bigl(G_{\mathrm{low}}(n) + G_{\mathrm{up}}(n)\bigr)\,
\frac{n}{\log^3 n},
\]
where $G_{\mathrm{low/up}}(n) = \prod_{\ell}
(1 - \nu_\ell/\ell)(1 - 1/\ell)^{-3}$, with
$\nu_\ell = \#\{0, -2, n \bmod \ell\}$ for $G_{\mathrm{low}}$ and
$\nu_\ell = \#\{0, 2, n \bmod \ell\}$ for $G_{\mathrm{up}}$ the number of
residue classes occupied by the pattern modulo $\ell$
($\nu_\ell \le 3 < \ell$ for $\ell \ge 5$; Corollary~\ref{cor:local}). We
write $\mathfrak{S}(n) = G_{\mathrm{low}}(n) + G_{\mathrm{up}}(n)$ for the
combined singular series, reserving the upright (S) for
Conjecture~\ref{conj:S}, and
\[
C_2 \;=\; \prod_{\ell > 2}\Bigl(1 - \frac{1}{(\ell-1)^2}\Bigr)
\;\approx\; 0.6602
\]
for the twin-prime constant. Applying the prediction to the triples
$(t,\, t+2,\, n-t)$ and $(t,\, t-2,\, n-t)$ with $t \le T$ gives for the
expected number of twin witnesses of $n$ below $T$
\[
W(n, T) \;\approx\; \mathfrak{S}(n)\,\frac{T}{\log^2 T\,\log n},
\]
up to bounded factors, the three logarithms corresponding to the primes $t$,
$t \pm 2$ and $n - t$. Solving $W(n,T) = 1$ gives, for the typical size of the
least witness,
\[
t_{\min}(n) \;\approx\; \frac{\log n\,(\log\log n)^2}{\mathfrak{S}(n)},
\]
since $\log T \approx \log\log n$ at that scale: this is $\log n$ up to a
slowly varying factor, matching the observed medians, which across the
$50{,}001$ blocks of the run to $10^{12}$ take only the six values
$13, 17, 19, 29, 31, 41$.

For the maximum over $n \le N$, model the witness events as independent,
so that $\Pr[t_{\min}(n) > T] \approx \exp(-W(n,T))$. The expected number
of even $n \le N$ with $t_{\min}(n) > T$ is then
\[
\sum_{n \le N} \exp\bigl(-W(n,T)\bigr)
\;=\; \sum_{n \le N}
\exp\!\Bigl(-\frac{\mathfrak{S}(n)\,T}{\log^2 T\,\log n}\Bigr),
\]
a sum dominated by the $n$ whose singular series is nearly minimal,
$\mathfrak{S}(n) \approx c_{\min} := \inf_n \mathfrak{S}(n)$. Equating the
dominant exponent with $\log N$, and using $\log T \approx 2\log\log N$ at
the resulting scale, yields
\[
\max_{n \le N} t_{\min}(n) \;\approx\;
\frac{4}{c_{\min}}\,(\log N)^2 (\log\log N)^2 .
\]
The exponent predicted is thus $2$, not $3$, up to the $(\log\log)^2$
factor. Against the exhaustive maxima of Section~\ref{sec:comp} the
right-hand side reads $9378$ against $14{,}449$ at $10^{11}$, $11{,}769$
against $14{,}549$ at $10^{12}$, $14{,}487$ against $20{,}747$ at $10^{13}$
and $17{,}542$ against $23{,}029$ at $10^{14}$: ratios of observed to
predicted of $1.54$, $1.24$, $1.43$ and $1.31$. The order is right throughout,
which is all a model of this crudeness can claim; the ratio oscillates and
shows no trend.

The extremal mechanism the model proposes is visible in the data. The infimum
is
\[
c_{\min} \;=\; 4 \cdot \frac{9}{8}
\prod_{\ell \ge 5}\Bigl(1 - \frac{3}{\ell}\Bigr)
\Bigl(1 - \frac{1}{\ell}\Bigr)^{-3}
\;=\; 2.8582487\ldots,
\]
attained exactly when $n \not\equiv 0 \pmod 3$ and no prime $\ell \ge 5$
divides $n(n+2)$ (for $n \equiv 1 \bmod 3$), resp.\ $n(n-2)$ (for
$n \equiv 2 \bmod 3$); for such $n$ one has the exact identity
\[
\frac{\mathfrak{S}(n)}{c_{\min}}
\;=\; \prod_{\substack{\ell \ge 5 \\ \ell \,\mid\, n(n \mp 2)}}
\frac{\ell - 2}{\ell - 3},
\]
the sign being that of the orientation forced by
Proposition~\ref{thm:orient}. When $3 \mid n$ both orientations survive and
$\mathfrak{S}(n)/c_{\min} \ge 2$, so extremal $n$ are never divisible by
$3$. The ten largest record witnesses all sit at a nearly minimal singular
series:

\begin{center}
\begin{tabular}{rccc}
\toprule
$n$ & $n \bmod 3$ & $\mathfrak{S}(n)/c_{\min}$ & $t_{\min}(n)$ \\
\midrule
$14{,}318$          & $2$ & $1.00098$ & $1021$ \\
$15{,}704{,}356$    & $1$ & $1.00000$ & $3299$ \\
$639{,}267{,}856$   & $1$ & $1.00455$ & $5849$ \\
$8{,}126{,}254{,}564$  & $1$ & $1.04728$ & $10{,}271$ \\
$77{,}739{,}196{,}118$ & $2$ & $1.01317$ & $14{,}449$ \\
$571{,}714{,}791{,}706$ & $1$ & $1.00014$ & $14{,}549$ \\
$6{,}527{,}240{,}154{,}856$ & $1$ & $1.00820$ & $20{,}747$ \\
$15{,}149{,}851{,}353{,}116$ & $2$ & $1.01330$ & $21{,}493$ \\
$29{,}330{,}655{,}440{,}536$ & $1$ & $1.01803$ & $21{,}599$ \\
$31{,}819{,}171{,}719{,}758$ & $2$ & $1.01556$ & $23{,}029$ \\
\bottomrule
\end{tabular}
\end{center}

All ten lie within $4.8\%$ of the minimum, whereas over $20{,}000$ random even
$n \in [2\cdot 10^9,\, 2 \cdot 10^{10}]$ the median of
$\mathfrak{S}(n)/c_{\min}$ is $1.61$ and only about $13\%$ lie within $4.8\%$
of it. These constants are computed in \texttt{verify\_heuristics.py}.

The last five rows carry more evidential weight than the others: the model,
with its constant $c_{\min}$ and its extremal mechanism, was committed before
the $10^{12}$ run, and the records found afterwards, one per decade to
$10^{14}$ together with the two intermediate ones above $10^{13}$, all land
within $1.9\%$ of the minimum, out of sample rather than fitted. None of this
proves the model: the $O(1)$ constant is uncalibrated, the prediction is still
exceeded at every decade, and three decades are three data points. The model
was merely at risk of being wrong here, and was not.

\begin{question}\label{q:main}
Is $t_{\min}(n) = O(\log^3 n)$? Is
$t_{\min}(n) = O(\log^{2+\varepsilon} n)$ for every $\varepsilon > 0$?
\end{question}

The heuristic answers both affirmatively, predicting
$\max_{n \le N} t_{\min}(n)/\log^3 N \asymp (\log\log N)^2/\log N \to 0$. The
single violation at $N = 14{,}318$ recorded in Section~\ref{sec:comp} is no
contradiction: under the model that ratio peaks at $\log N = e^2$, with value
$(4/c_{\min})\cdot 4/e^2 = 0.76$, and decays thereafter, whereas the observed
ratio peaks there at $1.17$; the shortfall is the uncalibrated $O(1)$ again.

\section{Computations}\label{sec:comp}

\textbf{Exhaustive verification.} (S) holds for every even
$6 \le n \le 10^{14}$, with no exception, and the maximum of
$t_{\min}(n)$ over the range is $23{,}029$, attained at
$n = 31{,}819{,}171{,}719{,}758$. This run used a C kernel: a
segmented bit-packed odd-only sieve with a shift-AND coverage scan, in
$5{,}000{,}001$ blocks of $2 \times 10^{7}$ integers whose recorded intervals
were checked to tile $[0,\, 10^{14}]$ contiguously, without gap or overlap, on
$16$ threads under OpenMP, $9.0$ hours of wall time on a laptop (i7-1360P).
The kernel caps the twin witnesses tried at $10^{7}$, and an even number
covered by no twin member below the cap is reported as a candidate exception
rather than silently missed; none was reported, so the run proves
$t_{\min}(n) \le 10^{7}$ for every even $n$ in range, with recorded maximum
$23{,}029$, far below the cap. The $10^{13}$ prefix agrees block for block
with our earlier separate run.

The segmented shift-and-test scan is the classical bit-array shift-and-OR
method of Sinisalo \cite{Sinisalo}, Deshouillers, te Riele and Saouter
\cite{DtRS} and Richstein \cite{Richstein}, with odd-only indexing throughout;
our kernels claim no algorithmic novelty. The earlier \texttt{numpy} kernels
remain as the historical double-check: the $10^{11}$ range was covered twice,
by two scans over a shared segmented sieve, with identical extremal statistics
on all $5001$ blocks, and $10^{12}$ by a single S-only pass in $50{,}001$
blocks. Since the sieve is common to both, the agreement tests the scans and
not the sieve.

The genuinely independent checks are two. The record witnesses were
re-verified with deterministic Miller--Rabin tests sharing no code with the
sieves: \texttt{verify\_records.py} confirms the four (S) records at
$10^{11}$, $10^{12}$, $10^{13}$ and $10^{14}$, together with the Dubner record
below, minimality included, by testing every smaller twin member. And the
records to $10^{13}$ were independently confirmed by Tom\'as Oliveira e Silva
(personal communication, August 2026), who verified (S) for all even
$n \le 10^{13}$ with an independent C program using libprimesieve, in about
$13$ hours of a single core, his record list agreeing exactly with ours. The
last decade, to $10^{14}$, rests on a single pass of our own kernel, with the
record witness re-verified as above.

The complete list of record values, that is of the even $n$ at which
$t_{\min}$ attains a value larger than at every smaller even number, is the
following $51$ pairs $(n,\, t_{\min}(n))$, read down the columns.

\begin{center}
\small
\setlength{\tabcolsep}{4pt}
\begin{tabular}{rrrrrr}
\toprule
$n$ & $t_{\min}$ & $n$ & $t_{\min}$ & $n$ & $t_{\min}$ \\
\midrule
$6$ & $3$ & $774{,}968$ & $1669$ & $3{,}760{,}188{,}526$ & $7457$ \\
$12$ & $5$ & $1{,}142{,}224$ & $1871$ & $6{,}082{,}985{,}476$ & $7559$ \\
$30$ & $7$ & $2{,}214{,}848$ & $2269$ & $8{,}126{,}254{,}564$ & $10{,}271$ \\
$98$ & $19$ & $6{,}137{,}438$ & $2689$ & $68{,}607{,}056{,}186$ & $11{,}059$ \\
$220$ & $29$ & $8{,}753{,}956$ & $2999$ & $77{,}739{,}196{,}118$ & $14{,}449$ \\
$308$ & $31$ & $15{,}704{,}356$ & $3299$ & $571{,}714{,}791{,}706$ & $14{,}549$ \\
$556$ & $107$ & $24{,}968{,}056$ & $3359$ & $1{,}821{,}067{,}383{,}866$ & $14{,}629$ \\
$796$ & $137$ & $31{,}480{,}496$ & $3469$ & $1{,}855{,}623{,}735{,}874$ & $15{,}647$ \\
$2776$ & $197$ & $36{,}612{,}068$ & $3919$ & $2{,}435{,}159{,}540{,}516$ & $16{,}189$ \\
$6908$ & $199$ & $59{,}892{,}398$ & $4051$ & $3{,}824{,}295{,}471{,}338$ & $16{,}651$ \\
$8138$ & $271$ & $65{,}712{,}758$ & $4639$ & $4{,}968{,}233{,}746{,}816$ & $17{,}207$ \\
$14{,}318$ & $1021$ & $241{,}609{,}594$ & $5417$ & $5{,}276{,}047{,}389{,}878$ & $17{,}419$ \\
$66{,}818$ & $1231$ & $453{,}267{,}964$ & $5477$ & $6{,}053{,}540{,}148{,}034$ & $18{,}311$ \\
$137{,}528$ & $1291$ & $639{,}267{,}856$ & $5849$ & $6{,}527{,}240{,}154{,}856$ & $20{,}747$ \\
$344{,}018$ & $1321$ & $1{,}200{,}364{,}448$ & $6271$ & $15{,}149{,}851{,}353{,}116$ & $21{,}493$ \\
$479{,}186$ & $1609$ & $1{,}311{,}113{,}198$ & $6451$ & $29{,}330{,}655{,}440{,}536$ & $21{,}599$ \\
$528{,}968$ & $1621$ & $3{,}104{,}283{,}256$ & $7349$ & $31{,}819{,}171{,}719{,}758$ & $23{,}029$ \\
\bottomrule
\end{tabular}
\end{center}

The exhaustive maxima at the decades are $3299$ to $2\times10^7$, $5849$ to
$10^9$, $10{,}271$ to $10^{10}$, $14{,}449$ to $10^{11}$, $14{,}549$ to
$10^{12}$, $20{,}747$ to $10^{13}$ and $23{,}029$ to $10^{14}$, the arguments
being read off the table above. Consistently
with Proposition~\ref{thm:orient}, the last of these has
$n \equiv 2 \pmod 3$ and $23{,}029$ is an upper member, $23{,}027$ being
prime. The record moved by only $+100$ across the decade to $10^{12}$, then by
$+6198$ and $+2282$ across the two decades after it.

The sole violation of $t_{\min}(n) \le \log^3 n$ in the verified range
is $t_{\min}(14{,}318) = 1021 > \log^3(14{,}318) = 876.3$, so that
$\max_{n \le N}t_{\min}(n)$ exceeds $\log^3 N$ exactly for
$14{,}318 \le N \le 23{,}612$. Beyond $10^{11}$ each decade closes the
question for the next, the maximum over $(10^{j},\, 10^{j+1}]$ lying below
$\log^3(10^{j}) \le \log^3 N$: $14{,}549 < 16{,}249 = \log^3(10^{11})$,
$20{,}747 < 21{,}096 = \log^3(10^{12})$ and
$23{,}029 < 26{,}821 = \log^3(10^{13})$, with $\log^3(10^{14}) = 33{,}499$
available for the decade after. At the decade endpoints the ratio of
$\max_{n \le N} t_{\min}(n)$ to $\log^3 N$ reads $0.63$ at $10^6$, $0.84$ at
$10^{10}$, $0.89$ at $10^{11}$, $0.69$ at $10^{12}$, $0.77$ at $10^{13}$ and
$0.69$ at $10^{14}$; the ratio is far from monotone, and what the data support
is that it stays well below $1$ throughout, not that it decreases. Windowed
distributional statistics of $t_{\min}$ over short ranges at each decade are
recorded in the accompanying repository.

\textbf{The Dubner side.} As a by-product the both-twin run, carried out to
$10^{11}$ with the \texttt{numpy} kernel and not repeated for any extension,
re-derives the $33$ terms $\ge 6$ of the Dubner exception list A007534 (all
$\le 4208$) and finds no further both-twin exception up to $10^{11}$,
re-verifying Dubner's conjecture and extending its verified range by a factor
of five: Dubner's own direct search \cite{Dubner} covered the multiples of $6$
up to $2 \times 10^{10}$, where it found the eleven exceptional centres of
Section~\ref{sec:mod3} and no others, which by the triple structure of that
section settles the even numbers to the same height. Dubner's conjecture is
\emph{not} verified above $10^{11}$ by the present computations, which settle
(S) alone there. The largest least \emph{Dubner} witness
(least twin $t$ with $n - t$ also twin) was $414{,}769$, at
$n = 35{,}984{,}652{,}098$; no larger occurs in the remaining two-thirds of
the range.

\begin{sloppypar}
The verification code and the per-block data are archived at
\href{https://doi.org/10.5281/zenodo.21744861}{doi:10.5281/zenodo.21744861}:
the \texttt{numpy} kernels \texttt{verify\_S.py},
\texttt{verify\_S\_segmented.py} and \texttt{verify\_S\_fast2.py} with its
driver \texttt{verify\_S\_seg\_fast.py}, the C kernel, the independent
checker \texttt{verify\_records.py}, and \texttt{verify\_heuristics.py},
which certifies the numerical constants of Section~\ref{sec:heur}.
\end{sloppypar}

\section{Related work and remarks}

The question of representing even numbers as sums of two twin members goes
back at least to Zwillinger \cite{Zw}; Dubner's conjecture and its
verification are discussed in \cite{Dubner, OEIS, Kourbatov}. The earliest
formulation of Conjecture~(S) we are aware of is the 2024 hypothesis of
Sahu \cite{Sahu}; it and the preprint \cite{NA} are discussed in the
introduction.
The unrestricted analogue of $t_{\min}$, the least prime $p$ for which
$n - p$ is also prime, is classical. Granville, van de Lune and te Riele
\cite{GLR} studied it computationally, and the verification of Oliveira e
Silva, Herzog and Pardi \cite{OeS} shows that every even
$n \le 4 \cdot 10^{18}$ has a Goldbach partition whose smaller prime is at
most $9781$, that maximum being attained at
$n = 3{,}325{,}581{,}707{,}333{,}960{,}528$
(OEIS \href{https://oeis.org/A025019}{A025019}, the record values, with their
arguments in \href{https://oeis.org/A025018}{A025018}). Against this, the
twin-restricted least witness already reaches $23{,}029$ by $10^{14}$; the
cost of restricting to $\T$ is quantitatively visible. Closer to
the counts $R_\T$ of Section~\ref{sec:heur} is M.~Barylski's
\href{https://oeis.org/A294185}{A294185} (2018), which counts the distinct
lesser twin primes in the Goldbach partitions of $2n$.

Any proof of (S) inherits the parity obstruction of sieve theory through its
twin-prime consequence, so (S) should be viewed as a unifying target rather
than an attackable intermediate. Theorem~\ref{thm:romanov} locates precisely
what would have to be supplied for the density version: not the distribution
of the twin members, only their number.

\section*{Acknowledgments}

I thank Tom\'as Oliveira e Silva for independently verifying (S) to $10^{13}$
with a program sharing no code with ours, for confirming the record list to
that height, and for his correspondence.

\end{document}